\documentclass[11pt]{amsart}
\usepackage{mathrsfs}
\usepackage{amssymb}

\usepackage{amscd}
\usepackage{amsmath, amssymb}
\usepackage{amsfonts}
\usepackage[colorlinks,linkcolor=blue,citecolor=blue, pdfstartview=FitH]{hyperref}

\usepackage{amscd}
\usepackage{easybmat}
\usepackage{mathrsfs}
\usepackage{amsfonts}
\usepackage{color}
\usepackage{pifont}
\usepackage{upgreek}
\usepackage{bm}
\usepackage{hyperref}
\usepackage{shorttoc}
\usepackage{amsmath,amstext,amsthm,a4,amssymb,amscd}
\usepackage[mathscr]{eucal}
\usepackage{mathrsfs}
\usepackage{epsf}

\numberwithin{equation}{section}

\def\p{\partial}
\def\o{\overline}
\def\b{\bar}
\def\mb{\mathbb}
\def\mc{\mathcal}
\def\n{\nabla}

\theoremstyle{plain}
\newtheorem{thm}{Theorem}[section]
\newtheorem{lemma}[thm]{Lemma}

\theoremstyle{definition}

\theoremstyle{definition}
\newtheorem{defn}[thm]{Definition}
\newcommand{\comment}[1]{}

\usepackage{fancyhdr}
\begin{document}

\title[
  The asymptotics of the $L^2$-curvature 
the  second variation of analytic torsion]{
  The asymptotics of the $L^2$-curvature 
  and the  second variation of analytic torsion on Teichm\"uller space
}
\makeatletter
\makeatother
\author{Xueyuan Wan}
\author{Genkai Zhang}

\address{Xueyuan Wan: Mathematical Sciences, Chalmers University of Technology, 41296 Gothenburg, Sweden}
\email{xwan@chalmers.se}

\address{Genkai Zhang: Mathematical Sciences, Chalmers University of Technology, 41296 Gothenburg, Sweden}
\email{genkai@chalmers.se}

\begin{abstract}
We consider
the relative canonical line bundle  $K_{\mc{X}/\mc{T}}$ and
a relatively ample line bundle $(L, e^{-\phi})$ 
over the total space $
\mc{X}\to \mc{T}$  of
fibration over the Teichm\"uller space 
by Riemann surfaces.
We consider the
case when the induced metric $\sqrt{-1}\p\b{\p}\phi|_{\mc{X}_y}$ on $\mc{X}_y$ has constant scalar curvature and  we obtain the curvature asymptotics of $L^2$-metric and Quillen metric of the direct image bundle $E^k=\pi_*(L^k+K_{\mc{X}/\mc{T}})$. As a consequence we prove that the second variation of analytic torsion $\tau_k(\b{\p})$ satisfies
\begin{align*}
 \p\b{\p}\log\tau_k(\b{\p})=o(k^{-l})
 \end{align*} 
 at the point $y\in\mc{T}$ for any $l\geq 0$ as $k\to\infty$. 
 \end{abstract}
\maketitle
\tableofcontents

\section{Introduction}

In the present paper we shall study the asymptotics of the second variation 
of analytic torsions for higher powers of a line bundle
for a family of Riemann surfaces. Consider first
a general holomorphic fibration  $\pi:\mc{X}\to \mathcal T$  
with fibers being
 $n$-dimensional compact manifolds, 
a relative ample line bundle $L$  
and  the relative canonical line bundle $K_{\mc{X}/\mathcal T}$
over $\mc{X}$.
Let $E^k=\pi_*(L^k+K_{\mc{X}/\mathcal T})$
be the  direct image bundle 
over $\mathcal T$. The holomorphic vector bundle $E^k$ 
is equipped with two natural metrics, 
the $L^2$-metric and Quillen metric (see Section \ref{sec1}). 
 In a recent preprint \cite{Wan} we 
prove that the second variation of analytic torsion satisfies
 \begin{align}\label{0.01}
 	\p\b{\p}\log\tau_k(\b{\p})=o(k^{n-1});
 \end{align}
 see \cite[Corollary 1.3]{Wan}. 
This is done by comparing the curvatures of these two metrics,
expanding the resolvent operator $(\Delta + k)^{-1}$ and by using Tian-Yau-Zelditch expansion of Bergman kernels on fibers. The
same result can also be deduced from a recent paper of Finski \cite{Finski}. 
It is then natural and interesting 
to ask whether the coefficients of the orders lower than $k^{n-1}$ are zero. One important
case of the above consideration is the fibration $\mathcal X$ over
 Teichm\"u{}ller space $\mathcal T$ by Riemann surfaces,
also called Teichm\"uller curve,  \cite{Ahlfors}.
It is now well-known
that the Bergman kernel expansion for line bundles $\mathcal L^k$
over Riemann surfaces  has only two terms, namly the linear term
 $c_1k$ and constant term $c_0$, the remaining term being 
exponentially decaying \cite{Liu}. We may expect
that the variation of the analytic torsion is also
decaying exponentially for $k\to \infty$. Indeed
when the line bundle $\mathcal L$  is the relative canonical
line bunde the exponential decaying property  is proved in \cite{Zhang};
the analytic torsion in this case is expressed in terms
of the Selberg zeta function \cite{Sarnak}
and can be studied by explicit computations.
In the present paper we will consider a 
general line bundle $\mathcal L$ over the 
the Teichm\"u{}ller curve and prove (\ref{0.01}) holds for any order.
In this case there is no explicit formula for the analytic torsion.
We describe below  more precisely our results and their proofs.

Let $L$ be a relative ample line bundle,  i.e. there exists a metric $\phi$ such that its curvature $$\sqrt{-1}\p\b{\p}\phi|_{\mc{X}_y}=\sqrt{-1}\phi_{v\b{v}}dv\wedge d\b{v}>0$$ on each fiber $\mc{X}_y:=\pi^{-1}(y)$ for any $y\in \mc{T}$, and let $K_{\mc{X}/\mc{T}}:=K_{\mc{X}}-\pi^*K_{\mc{T}}$ denote the relative canonical line bundle. We will consider the following direct image bundle 
\begin{align}
E^k=\pi_*(L^k+K_{\mc{X}/\mc{T}})
\end{align}
over the Teichm\"uller space $\mc{T}$.

Let $D_y=\b{\p}_y+\b{\p}^*_y$  be the Dirac operator 
acting on
$A^{0, 1}(\mc{X}_y, L^{k}+K_{\mc{X}/\mc{T}})$
of $(0, 1)$-forms.
 For any $b>0$, denote by $D^{(b,+\infty)}$ the restriction of $D$ on the sum of eigenspaces of $A^{0,1}(\mc{X}_y,  L^{k}+K_{\mc{X}/\mc{T}})$ for eigenvalues in $(b,+\infty)$. Then the (Ray-Singer) analytic torsion is defined by
\begin{equation*}
\tau_k(\b{\p})=\tau_k(\b{\p}^{(b,+\infty)})=\left(\det ((D^{(b,+\infty)})^2)\right)^{1/2}\end{equation*}
and is a positive smooth function on $\mc{T}$.
Here $b$ is a constant less than all positive eigenvalues of $D$ (see  
Definition \ref{torsiondef}).

The analytic torsion and its second variation on Teichm\"uller space have been studied in details by \cite[Theorem 3.10 and Theorem 5.8]{Fay}. In this paper, we prove 
\begin{thm}\label{main theorem}
Let $\pi: \mc{X}\to \mc{T}$ be the holomorphic fibration over Teichm\"uller space $\mc{T}$. Suppose that the induced metric $\sqrt{-1}\p\b{\p}\phi|_{\mc{X}_y}$ on $\mc{X}_y$ has constant scalar curvature. Then
	 \begin{align}\label{second variation}
 \p\b{\p}\log\tau_k(\b{\p})=o(k^{-l})
 \end{align} 
 at any  point $y\in \mc{T}$ for any $l\geq 0$ as $k\to\infty$. 
\end{thm}
Here the asymptotic (\ref{second variation}) denotes $(\p\b{\p}\log\tau_k(\b{\p}))(\zeta,\o{\zeta})=o(k^{-l})$ for any vector $\zeta\in T_y\mc{T}$. 
%

Recall Theorem \ref{Bis} (see below) that the
Quillen metric on $\mathcal T$
is defined by the analytic torsion for the the fibration.
In the papers \cite{Bismut1,Bismut2, Bismut3}, 
J.-M. Bismut, H. Gillet and C. Soul\'{e} computed the curvature of Quillen metric for a locally K\"ahler family and obtained the differential form version of Grothendieck-Riemann-Roch Theorem. Moreover, 
 they proved that as a holomorphic bundle,
$$\lambda_y\cong \bigotimes_{i\geq 0}\det H^i(\mc{X}_y, L^{k}+K_{\mc{X}/\mc{T}})^{(-1)^{i+1}}.$$ 
 By Kodaira vanishing theorem, $H^i(\mc{X}_y,
 K_{\mc{X}/\mc{T}}+L^k)=0$ for all $i\geq 1$, thus $$\lambda\cong
 (\det E^k)^{-1}.$$
Let  $\det \|\bullet\|_k$ denote  the natural induced $L^2$-metric on line bundle $\lambda^{-1}=\det E^k$.  Then it follows from
(\ref{metric consist})
that
\begin{align}\label{0.10}\det \|\bullet\|_k^2=((|\bullet|^b)^2)^*\end{align}
for $b>0$ a sufficiently small constant,  
where  $((|\bullet|^b)^2)^*$ denotes the dual of the $L^2$-metric $(|\bullet|^b)^2$.

The Chern forms of the $L^2$-metric
has been studied intensively and 
 Berndtsson \cite{Bern2, Bern3, Bern4}
has found the curvature of the vector bundle,
 \begin{align*}
\langle\sqrt{-1}\Theta^{E_k}u,u\rangle=\int_{\mc{X}/M}kc(\phi)|u|^2e^{-k\phi}+k\langle (\Delta'+k)^{-1} i_{\mu_{\alpha}}u, i_{\mu_{\beta}}u\rangle\sqrt{-1}dz^{\alpha}\wedge d\b{z}^{\beta},	
\end{align*}
where the definitions of $c(\phi)$, $\mu_{\alpha}$ and $\Delta'$ are
given  in Theorem \ref{BF1}.
To prove Theorem \ref{main theorem} we shall find the expansion of
 the  curvature of $E^k$. 
\begin{thm}\label{main theorem 1}
For any vector $\zeta=\zeta^{\alpha}\frac{\p}{\p z^{\alpha}}\in T_y\mc{T}$, we have 
	\begin{align}\label{0.7}
\begin{split}
	& -\sqrt{-1}c_1(E^k,\|\bullet\|_k)(\zeta,\b{\zeta})=\frac{6k^2-6k\rho+\rho^2}{24\pi^2(-\rho)}\|\mu\|^2+o(k^{-l})
\end{split}	
\end{align}
for any $l\geq 0$, where $\mu=-\p_{\b{v}}(\phi_{\alpha\b{v}}(\phi_{v\b{v}})^{-1})\zeta^{\alpha} d\b{v}\otimes \frac{\p}{\p v}$ and $$\|\mu\|^2=\int_{\mc{X}_y}|\p_{\b{v}}(\phi_{\alpha\b{v}}(\phi_{v\b{v}})^{-1})\zeta^{\alpha}|^2\sqrt{-1}\p\b{\p}\phi.$$
\end{thm}

The Quillen metric $\|\bullet\|_Q$ on the determinant line $\lambda$ (see Definition \ref{line bundle}) is patched by the $L^2$-metric $|\bullet|^b$ on $\lambda^b$ (see (\ref{lambdab})) and the analytic torsion $\tau_k(\b{\p})$, i.e.
\begin{align}\label{0.9}
	\|\bullet\|^b=|\bullet|^b\tau_k(\b{\p}),
\end{align}
where $b>0$ is a sufficiently small constant.
From \cite[Proposition 3.9]{Wan}, we obtain the curvature of Quillen metric.
\begin{thm}\label{main theorem 2} For $\zeta\in T_y\mc{T}$, we have
	\begin{align}
(\sqrt{-1}c_1(\lambda,\|\bullet\|_Q))(\zeta,\b{\zeta})=\frac{6k^2-6k\rho+\rho^2}{24\pi^2(-\rho)}\|\mu\|^2.
\end{align}
\end{thm}

Using (\ref{0.10}) and (\ref{0.9}) we have furthermore
\begin{align}
\frac{\sqrt{-1}}{2\pi}\p\b{\p}\log (\tau_k(\b{\p}))^2
=-c_1(\lambda,\|\bullet\|_Q)-c_1(E^k, \|\bullet\|_k).
\end{align}
Theorem \ref{main theorem} is then
 an immediate consequence
of  Theorem \ref{main theorem 1} and Theorem \ref{main theorem 2}.

We mention finally that our result states that the second variation of
analytic torsion decays faster than any $k^{-l}$, and it would be interesting
to know if it is decaying exponentially as $e^{-k c}$.

This article is organized as follows. In Section \ref{sec1}, we fix the notation and recall some definitions and facts on analytic torsion, Quillen metric, Berndtsson's curvature formula and Bergman kernel on Riemann surface. In Section \ref{sec2}, we  find
 the expansion of $c_1(E^k, \|\bullet\|_k)$ and prove Theorem \ref{main theorem 1}. We also give the expression of $-c_1(\lambda, \|\bullet\|_Q)$ and prove Theorem \ref{main theorem 2}. By comparing with their expansions, we prove Theorem \ref{main theorem}.

We would like to thank
 Bo Berndtsson for many insightful
 discussions on the curvature formula of direct image bundles and
 Miroslav Englis for careful explanation of Bergman kernel
expansions on Riemann surfaces, 
	
\section{Preliminaries}\label{sec1}

\subsection{Analytic torsion and Quillen metric}\label{sub1}
We start with  the rather general setup of holomorphic fibrations
and  specify them later to the case of the fibration by Riemann surfaces
over Teichm\"u{}ller space. 
The definitions and results in this subsection
can be found in  \cite{BGV, Bismut1, Bismut2, Bismut3,  Ma1, Ray}.

 Let $\pi:\mc{X}\to M$ be a  proper holomorphic mapping between complex manifolds $\mc{X}$ and $M$, 
$(F,h_F)$  a holomorphic Hermitian vector bundle on
 $\mc{X}$, $\n^{F}$ the corresponding Chern connection, and
$R^{F}=(\n^{F})^2$ its curvature. 
For any $y\in M$, let $\mc{X}_y=\pi^{-1}(y)$ be the fiber over $y$ with 
 K\"ahler metric $g^{\mc{X}_y}$ depending smoothly on $y$. The fibers are assumed
to be compact.

The operator $D_y=\b{\p}_y+\b{\p}^*_y$ acts on the fiber $A^{0,*}(\mc{X}_y,F)$.
 For $b>0$, let $K^{b,p}_y$ be the sum of the eigenspaces of the operator  $D^2_y$ acting on $A^{0,p}(\mc{X}_y,F)$ for eigenvalues $<b$.
Let $U^b$ be the open set
$U^b=\{y\in M; b\not\in \text{Spec}D^2_y\}.$
Set
\begin{align*}
K^{b,+}=\bigoplus_{p\,\text{even}}K^{b,p},\quad K^{b,+}=\bigoplus_{p\,\text{odd}}K^{b,p},\quad K^b=K^{b,+}\oplus K^{b,-}.	
\end{align*}
Define the following line bundle $\lambda^b$ on $U^b$ by
\begin{align}\label{lambdab}\lambda^b=(\det K^{b,0})^{-1}\otimes (\det K^{b,1})\otimes (\det K^{b,2})^{-1}\otimes\cdots.\end{align}
For $0<b<c$, if $y\in U^b\cap U^c$, let $K^{(b,c),p}_y$ be the sum of
eigenspaces of $D^2_y$ in $E^p_y$ for eigenvalues $\mu$ such that
$b<\mu<c$.
Define similarly $K^{(b,c),+}_y$, $K^{(b,c),-}_y$,  $K^{(b,c)}_y$ and
and $\lambda^{(b,c)}$. Let $\b{\p}^{(b,c)}$ and $D^{(b,c)}$ be the restriction of $\b{\p}$ and $D$ to $K^{(b,c)}$,
and  $D^{(b,c)}_{\pm}$  the restriction of $D$ to $K^{(b,c),\pm}$.
The bundle  $\lambda^{(b,c)}$ has a canonical non-zero section $T(\b{\p}^{(b,c)})$ which is smooth on $U^b\cap U^c$ (see \cite[Definition 1.1]{Bismut1}). For $0<b<c$, over $U^b\cap U^c$, we have the $C^{\infty}$ identifications
$
\lambda^c=\lambda^b\otimes \lambda^{(b,c)},	
$ which is given by the following $C^{\infty}$ map
\begin{align}\label{tran}
s\in \lambda^b\mapsto s\otimes T(\b{\p}^{(b,c)})\in \lambda^c.
\end{align}

 \begin{defn}[{\cite[Def. 1.1]{Bismut3}}]\label{line bundle}
 {\it The $C^{\infty}$ line bundle  $\lambda$ over $M$ 
is $\{(U^b, \lambda^b)\}$  with the transition functions (\ref{tran}) on $U^b\cap U^c$.}
 \end{defn}
The analytic torsion  was  introduced by
Ray and Singer \cite{Ray}.
\begin{defn}\label{torsiondef}
	{\it The analytic torsion} $\tau(\b{\p}^{(b,c)})$ is  defined as the positive real number
\begin{align*}
\tau(\b{\p}^{(b,c)})=\left((\det (D_1^{(b,c)})^2)(\det (D_2^{(b,c)})^2)^{-2}(\det (D_3^{(b,c)})^2)^3\cdots\right)^{1/2},
\end{align*}
where $D^{(b,c)}_p$ is the restriction of $D$ to $K^{(b,c), p}$, $1\leq p\leq n$.  If $b$ is a small constant less than all positive eigenvalues of $D^2_y$, we denote
$\tau(\b{\p}):=\tau(\b{\p}^{(b,+\infty)}).$
\end{defn}
Let $\|\bullet\|^b$ denote the following metric on the line bundle $(\lambda^b, U^b)$,
\begin{align}\label{torsion}
\|\bullet\|^b=|\bullet|^b\tau_y(\b{\p}^{(b,+\infty)}),
\end{align}
 where $|\bullet|^b$ is the standard $L^2$-metric.
The definition of Quillen metric $\|\bullet\|_Q$ and  its
Chern form $c_1(\lambda, \|\bullet\|_Q)$ are given by the following theorem.
\begin{thm}[\cite{Bismut1, Bismut2, Bismut3}]\label{Bis}
The metrics $\|\bullet\|^b$ on $(\lambda^b, U^b)$ patch into a smooth Hermitian metric $\|\bullet\|_Q$ on the holomorphic line bundle $\lambda$.	
The Chern form of Hermitian line bundle $(\lambda, \|\bullet\|_Q)$ is
\begin{align}\label{Chern forms}
c_1(\lambda, \|\bullet\|_Q)=-\left\{\int_{\mc{X}/M} Td\left(\frac{-R^{T_{\mc{X}/M}}}{2\pi i}\right)Tr\left[\exp\left(\frac{-R^{F}}{2\pi i}\right)\right]\right\}^{(1,1)}.
\end{align}
\end{thm}

The Knudsen-Mumford determinant is defined by
\begin{align*}
\lambda^{KM}=(\det R\pi_*F)^{-1}.	
\end{align*}
On each fiber it is given by
$\lambda^{KM}_y=\bigotimes_{i\geq 0}\det H^i(\mc{X}_y,F)^{(-1)^{i+1}}.	$

\begin{thm}[\cite{Bismut1, Bismut2, Bismut3}]\label{Bis11}
Assume that $\pi$ is locally K\"ahler. Then the identification of the fibers $\lambda_y\cong \lambda^{KM}_y$ defines a holomorphic isomorphism of line bundles $\lambda\cong \lambda^{KM}$. The Chern form of the Quillen metric on $\lambda\cong \lambda^{KM}$ is given by
(\ref{Chern forms}).
\end{thm}
Here locally K\"ahler means that there is an open covering $\mathscr{U}$ of $M$ such that if $U\in \mathscr{U}$, $\pi^{-1}(U)$ admits a K\"ahler metric.

\subsection{Berndtsson's curvature formula of $L^2$-metric}\label{Subsection2}

We refer \cite{Bern2, Bern3, Bern4} and references therein for the proof and
background.

 Let $\pi:\mc{X}\to M$ be a holomorphic fibration with compact fibres and $L$ a relative ample line bundle over $\mc{X}$. We denote by
$(z;v)$ a local admissible holomorphic coordinate system of $\mc{X}$ with $\pi(z;v)=z$, where $z=\{z^\alpha\}_{1\leq \alpha\leq \dim M}$, $v=\{v^i\}_{1\leq i\leq \dim\mc{X}-\dim M}$ are the local coordinates of $M$ and fibers, respectively.

Let $\phi$ be a metric of $L$ such that 
$
(\sqrt{-1}\p\b{\p}\phi)|_{\mc{X}_y}>0	
$
for any point $y\in M$. Set 
\begin{align}\label{horizontal}
  \frac{\delta}{\delta z^{\alpha}}:=\frac{\p}{\p z^{\alpha}}-\phi_{\alpha\b{j}}\phi^{\b{j}k}\frac{\p}{\p v^k},\quad 1\leq \alpha\leq \dim M.
\end{align}
Here $\phi_{\alpha\b{j}}:=\p_{z^{\alpha}}\p_{\b{v}^j}\phi$, $(\phi^{\b{j}k})$ denotes the inverse of the matrix $(\p_{k}\p_{\b{j}}\phi)$. The geodesic curvature $c(\phi)$ is defined by
\begin{align*}
  c(\phi)=\left(\phi_{\alpha\b{\beta}}-\phi_{\alpha\b{j}}\phi^{i\b{j}}\phi_{i\b{\beta}}\right)\sqrt{-1} dz^{\alpha}\wedge d\b{z}^{\beta},
\end{align*}
which is a well-defined real $(1,1)$-form on $\mc X$. Let $\{dz^{\alpha};\delta v^k\}$
denote the dual frame of $\left\{\frac{\delta}{\delta z^{\alpha}}; \frac{\p}{\p v^i}\right\}$. 
The form $\sqrt{-1}\p\b{\p}\phi$ has
the following decomposition \cite[Lemma 1.1]{Feng}
  \begin{align}\label{dec}
    \sqrt{-1}\p\b{\p}\phi=c(\phi)+\sqrt{-1}\phi_{i\b{j}}\delta v^i\wedge \delta \b{v}^j.
  \end{align}
Consider the direct image bundle $E:=\pi_*(K_{\mc{X}/M}+L)$ with the 
natural $L^2$-metric, \cite{Bern2, Bern3, Bern4},
\begin{align}\label{L2 metric}
\|u\|^2:=\int_{\mc{X}_y}|u|^2e^{-\phi}. 	
\end{align}
for any $u=u' dv\otimes e\in E_y$, where $e$ is a local holomorphic frame of $L|_{\mc X}$, $dv=dv^1\wedge \cdots\wedge dv^n$. Here 
$$|u|^2e^{-\phi}:=(\sqrt{-1})^{n^2}|u'|^2 |e|^2dv\wedge d\b{v}=(\sqrt{-1})^{n^2}|u'|^2 e^{-\phi}dv\wedge d\b{v}.$$ We denote
\begin{align*}
\mu_{\alpha}=-\frac{\p}{\p \b{v}^l}\left(\phi_{\alpha\b{j}}\phi^{\b{j}i}\right)d\b{v}^l\otimes \frac{\p}{\p v^i}.
\end{align*}
 The following theorem was proved by Berndtsson in \cite[Theorem 1.2]{Bern4}, its proof  can also be found in \cite[Theorem 3.1]{Feng}.

\begin{thm}[\cite{Bern4}]\label{BF1}
\label{thm4} For any $y\in M$ 
the curvature 
 $\langle \Theta^{E}u, u\rangle$, $ u\in E_{y}$, 
  of the Chern connection on $E$ with the $L^2$-metric 
is given by
\begin{align*}
\langle \sqrt{-1}\Theta^{E}u,u\rangle=\int_{\mc{X}_y}c(\phi)|u|^2e^{-\phi}+\langle(1+\Delta')^{-1}i_{\mu_{\alpha}}u,i_{\mu_{\beta}}u\rangle\sqrt{-1}dz^{\alpha}\wedge d\b{z}^{\beta}.
\end{align*}
Here $\Delta'=\n'\n'^*+\n'^*\n$ is the Laplacian on $L|_{\mc{X}_y}$-valued forms on $\mc{X}_y$ defined by the $(1,0)$-part of the Chern connection on $L|_{\mc{X}_y}$.
\end{thm}

We replace now the Hermitian line bundle $(L,e^{-\phi})$ by $(L^k,e^{-k\phi})$, 
and consider the corresponding direct image bundle $E^k:=\pi_*(L^k+K_{\mc{X}/M})$. Let $\n'^*_k$ (resp. $\n'^*$) be the adjoint operator of  $\n'$ with respect to $(L^k,e^{-k\phi})$ and $(X, k\omega=k\sqrt{-1}\p\b{\p}\phi)$ (resp. $(X, \omega=\sqrt{-1}\p\b{\p}\phi)$). We have
\begin{align}
\n'^*=\sqrt{-1}[\Lambda_{k\omega},\n']=\frac{1}{k}\sqrt{-1}[\Lambda_{\omega},\n']=\frac{1}{k}\n'^*.	
\end{align}
Hence
\begin{align}\label{delt1}
\Delta'_k=\n'^*_k\n'+\n'\n'^*_k=\frac{1}{k}\Delta'.	
\end{align}
 From Theorem \ref{BF1} and (\ref{delt1}), the curvature of $L^2$-metric (see (\ref{L2 metric})) on $E^k$ is given by 
\begin{align}\label{L2 curvature}
\begin{split}
\langle \sqrt{-1}\Theta^{E^k}u,u\rangle &=\int_{\mc{X}_y}c(k\phi)|u|^2e^{-k\phi}+\langle (1+\Delta'_k)^{-1}i_{\mu_{\alpha}}u,i_{\mu_{\beta}}u\rangle_{k\omega}\sqrt{-1}dz^{\alpha}\wedge d\b{z}^{\beta}\\
&=\int_{\mc{X}_y}kc(\phi)|u|^2e^{-k\phi}+k\langle (k+\Delta')^{-1}i_{\mu_{\alpha}}u,i_{\mu_{\beta}}u\rangle\sqrt{-1}dz^{\alpha}\wedge d\b{z}^{\beta}
\end{split}
\end{align}
for any element $u$ of $E^k_y$.

\subsection{Bergman Kernel on Riemann surface}

Let $M$ be an compact complex K\"ahler manifold with an ample line bundle $L$ over $M$. Let $g$ be the K\"ahler metric on $M$ corresponding to the K\"ahler form $\omega_g=\text{Ric}(h)$ for some positive curvature Hermitian metric $h$ on $L$. The metric $h$ induces a metric $h_k$ on $L^k$.
Let $\{S_0,\cdots, S_{d_k-1}\}$ be an orthonormal basis of the space $H^0(M, L^k)$ with respect to the inner 
$$(S,T)=\int_M \langle S(x), T(x)\rangle_{h_k}dV_g,$$
where $d_k=\dim H^0(M, L^k)$. Then the diagonal of Bergman kernel is given by
\begin{align}
	\sum_{i=0}^{d_k-1}|S_i(x)|^2_{h_k}.
\end{align}

The Tian-Yau-Zelditch expansion of Bergman kernel 
has been extensively studied. For Riemann surfaces it has
some particular nature in that the expansion has only two terms;
more precisely we have
the following 

\begin{thm}[{\cite[Theorem 1.1]{Liu}}]\label{TYZ}
	Let $M$ be a regular compact Riemann surface and $K_M$ be the canonical line bundle endowed with a Hermitian metric $h$ such that the curvature $\text{Ric}(h)$ of $h$ defines a K\"ahler metric $g$ on $M$. Suppose that the metric $g$ has constant scalar curvature $\rho$. Then there is a complete asymptotic expansion:
	\begin{align}
		\sum_{i=0}^{d_k-1}\|S_i(x)\|^2_{h_k}\sim k(1+\frac{\rho}{2k})+O\left(e^{-\frac{(\log k)^2}{8}}\right),
	\end{align}
where $\{S_0,\cdots, S_{d_k-1}\}$ is an orthonormal basis for $H^0(M, K_M^k)$ for some $k>\max\{e^{20\sqrt{5}}+2|\rho|, |\rho|^{4/3}, \frac{1}{\delta}, \sqrt{\frac{2}{|\rho|}}\}$, where $\delta$ is the injective radius at $x_0$.
\end{thm}

We note that the expansion holds
also for the bundle $L^m +K_M$ where $L$ is any ample line bundle such that its curvature gives a K\"ahler metric with constant scalar curvature.
Indeed the same proof there works also for this case; alternatively one may
argue abstractly that the expansion is determined by the curvature of $L$.
(Presumably the above expansion can  be proved using
the more elementary method in \cite{Englis}.)

\section{The second variation of analytic torsion}\label{sec2}

Let $\mc{T}$ be the Teichm\"uller space of Riemann surface of genus
$g\geq 2$. Let $\pi: \mc{X}\to \mc{T}$ be the holomorphic fibration
of the Teichm\"u{}ller curve over
$\mc{T}$,  the fiber $\mc{X}_y:=\pi^{-1}(y)$ being exactly the Riemann
surface given by the complex structure $y\in\mc{T}$;
see \cite{Ahlfors}.
 Let $L$ be a relative ample line bundle over $\mc{X}$, namely there exists a metric $\phi$ of $L$ such that the curvature $\sqrt{-1}\p\b{\p}\phi|_{\mc{X}_y}>0$, this implies that $\pi:\mc{X}\to \mc{T}$ is a local K\"ahler fibration. 
Denote 
\begin{align}
\omega=\sqrt{-1}\p\b{\p}\phi.	
\end{align}
We take $(z^1,\cdots, z^m, v)$  a local admissible coordinate system
of $\mc{X}$
as in the Subsection \ref{Subsection2}.  
Then $\omega|_{\mc{X}_y}=\sqrt{-1}\phi_{v\b{v}}dv\wedge d\b{v}$ gives a K\"ahler metric on $\mc{X}_y$, $\phi_{v\b{v}}:=\frac{\p^2\phi}{\p v\p\b{v}}$. The scalar curvature is defined by 
\begin{align}\label{rho}
\rho=-\frac{1}{\phi_{v\b{v}}}\p_v\p_{\b{v}}\log\phi_{v\b{v}}.
\end{align}

Now we assume that the scalar curvature $\rho$ is a constant.
Up to a constant we  can
 take $\phi$ such that
\begin{align}\label{1.1}
e^{-\rho\phi}=\phi_{v\b{v}}.	
\end{align}
In particular,  $-\rho=\int_{\mc{X}_y}c_1(K_{\mc{X}_y})/\int_{\mc{X}_y}c_1(L)$ is a positive rational number. 
Let $K_{\mc{X}/\mc{T}}:=K_{\mc{X}}-\pi^*K_{\mc{T}}$ denote the
relative canonical line bundle.
Consider the following direct image bundle over $\mc{T}$
\begin{align}
E^k:=\pi_*(L^k+K_{\mc{X}/\mc{T}}),
\end{align}
for any integer $k\geq 1$.
The operator $D_y=\b{\p}_y+\b{\p}^*_y$ acts on  $A^{0,*}(\mc{X}_y, L^k+K_{\mc{X}/\mc{T}})$.
Take  a constant $b>0$  smaller than all the positive eigenvalues of $D_y$. Then
$
K^{b,p}_y\cong H^p(\mc{X}_y, K_{\mc{X}_y}+L^k).
$
Furthermore by Kodaira vanishing theorem,
  \begin{align*}
  K^{b,0}_y\cong H^{0}(\mc{X}_y,L^k+K_{\mc{X}_y})\cong \pi_*(L^k+K_{\mc{X}/\mc{T}})_y  \quad K_y^{b,p}=0, \quad \text{for}\quad p\geq 1, 	
  \end{align*}
consequently
  \begin{align}\label{3.1}
  \lambda^b=(\det \pi_*(L^k+K_{\mc{X}_y}))^{-1}.	
  \end{align}
Since the metric $\phi$ induces a metric  $(\det\phi)^{-1}:=(\phi_{v\b{v}})^{-1}$ on $K_{\mc{X}/\mc{T}}$, by (\ref{L2 metric}), we have
\begin{align}\label{metric consist}
|u|^2e^{-k\phi}=\sqrt{-1}|u'|^2e^{-k\phi}dv\wedge d\b{v}=|u'|^2e^{-k\phi}(\det\phi)^{-1}\omega=|u|^2_{L^2}\omega,
\end{align}
that is, the $L^2$-metric $\|\bullet\|_k$ on $\pi_*(L^k+K_{\mc{X}/\mc{T}})$ given by (\ref{L2 metric}) coincides with the standard $L^2$-metric on $\pi_*(L^k+K_{\mc{X}/\mc{T}})$ induced by $(\mc{X}_y, \omega|_y)$, $(K_{\mc{X}_y}, (\det\phi)^{-1})$ and $(L,e^{-\phi})$.
Thus the $L^2$-metric $(|\bullet|^b)^2$ is dual to the determinant of
the metric $\|\bullet\|^2$.  Using the definition (\ref{torsion}) we have
then
\begin{align}\label{torsion1}
(\|\bullet\|^b)^2=(|\bullet|^b)^2(\tau_k(\b{\p}^{(b,+\infty)}))^2=(\det\|\bullet\|^2_k)^* (\tau_k(\b{\p}))^2,	
\end{align}
for $b>0$ small enough, where $\tau_k(\b{\p})=\tau_{k}(\b{\p}^{(b,+\infty)})$ is the analytic torsion associated with $(\mc{X}, \omega=\sqrt{-1}\p\b{\p}\phi))$ and $(L^k, e^{-k\phi})$. Therefore,
\begin{align}\label{Torsion}
\begin{split}
\frac{\sqrt{-1}}{2\pi}\p\b{\p}\log (\tau_k(\b{\p}))^2
&=-c_1(\lambda,\|\bullet\|_Q)-c_1(E^k, \|\bullet\|_k).
\end{split}
\end{align}

\subsection{The curvature of $L^2$-metric} In this subsection we will find
the expansion of  the first Chern form $c_1(E^k, \|\bullet\|_k)$. 

From (\ref{L2 curvature}), the curvature of $L^2$-metric is 
\begin{align}
\begin{split}
\langle \sqrt{-1}\Theta^{E^k}u,u\rangle 
&=\int_{\mc{X}_y}kc(\phi)|u|^2e^{-k\phi}+k\langle (k+\Delta')^{-1}i_{\mu_{\alpha}}u,i_{\mu_{\beta}}u\rangle\sqrt{-1}dz^{\alpha}\wedge d\b{z}^{\beta}
\end{split}
\end{align}
for any element $u$ of $E^k_y$.
For any tangent vector $\zeta=\zeta^{\alpha}\frac{\p}{\p z^{\alpha}}\in T_y\mc{T}$
\begin{align}\label{1.2}
\begin{split}
	&\quad(-\sqrt{-1})c_1(E^k, \|\bullet\|_k)(\zeta,\b{\zeta})=\frac{1}{2\pi}\sum_{j=1}^{d_k}\langle\Theta^{E^k}u_j,u_j\rangle(\zeta,\b{\zeta})\\
	&=-\frac{\sqrt{-1}}{2\pi}\left(\int_{\mc{X}_y}kc(\phi)\sum_{j=1}^{d_k}|u_j|^2e^{-k\phi}\right)(\zeta,\b{\zeta})
	+\frac{1}{2\pi}\sum_{j=1}^{d_k}k\langle (k+\Delta')^{-1}i_{\mu}u_j,i_{\mu}u_j\rangle,
	\end{split}
\end{align}
where  $d_k:=\dim H^0(\mc{X}_y, (L^k+K_{\mc{X}/\mc{T}})|_{\mc{X}_y})$, $\{u_j\}_{j=1}^{d_k}$ is a orthonormal basis of $H^0(\mc{X}_y, (L^k+K_{\mc{X}/\mc{T}})|_{\mc{X}_y})$ and
\begin{align}\label{definition of mu}
\mu=-\p_{\b{v}}(\phi_{\alpha\b{v}}(\phi_{v\b{v}})^{-1})\zeta^{\alpha} d\b{v}\otimes \frac{\p}{\p v}.	
\end{align}
We shall now find expansion of the two terms in RHS of (\ref{1.2}) in $k$.

For any $y\in\mc{T}$, $\mc{X}_y$ is a Riemann surface and the metric $\omega|_{\mc{X}_y}$ is a K\"ahler metric with constant scalar curvature $\rho$. By (\ref{1.1}), the curvature operator $R^*$ of cotangent bundle is 
\begin{align}
\begin{split}
R^*&:=R^{v\b{v}}_{v\b{v}}i_{\frac{\p}{\p v}}dv\wedge d\b{v}\wedge i_{\frac{\p}{\p \b{v}}}\\
&:=(-\p_v\p_{\b{v}}(\phi_{v\b{v}})^{-1}+\p_v(\phi_{v\b{v}})^{-1}\p_{\b{v}}(\phi_{v\b{v}})^{-1}\phi_{v\b{v}})i_{\frac{\p}{\p v}}dv\wedge d\b{v}\wedge i_{\frac{\p}{\p \b{v}}}\\
&=-e^{\rho\phi}\p_v\p_{\b{v}}\log e^{\rho\phi}\cdot \text{Id}\\
&=-\rho\cdot \text{Id}
\end{split}
\end{align}
when acting on the space $A^{0,1}(\mc{X}_y, L^k)$. Thus by \cite[Lemma 3.2]{Wan}, for any $\alpha\in A^{0,1}(\mc{X}_y, L^k)$, 
\begin{align}\label{0.1}
\begin{split}
(k+\rho-\n'^*\n')\alpha&=(k-R^*-\n'^*\n')\alpha\\
&=(dv)^*\n'(\n_{\b{v}}\alpha)\\
&=\phi^{-1}_{v\b{v}}i_{\frac{\p}{\p v}}(\p-k\p\phi)\n_{\b{v}}\alpha\\
&=e^{\rho\phi}\n_{v}\n_{\b{v}}\alpha.	
\end{split}
\end{align}
Here $\n'^*$ denotes the adjoint operator of the $(1,0)$-part $\n'$ of Chern connection, and we have used the following notations:
\begin{align}\label{notation1}
\n_{l\b{v}}:=\p_{\b{v}}+l\rho\phi_{\b{v}},\quad \n_v:=\p_v- k\phi_v
\end{align}
for any integer $l$.

\begin{lemma}\label{lemma1} For any $a, b\in \mb{R}$, we have 
\begin{align}
e^{\rho\phi}[\p_{\b{v}}-a\phi_{\b{v}}, \p_v-b\phi_v]=a-b.	
\end{align}
\end{lemma}
\begin{proof}
By a direct computation we have
\begin{align*}
	e^{\rho\phi}[\p_{\b{v}}-a\phi_{\b{v}}, \p_v-b\phi_v]=e^{\rho\phi}(-b\phi_{v\b{v}}+a\phi_{v\b{v}})=a-b,
\end{align*}
	where the last equality follows from (\ref{1.1}).
\end{proof}
Using the lemma we have then
\begin{align}
e^{\rho\phi}[\n_{\b{v}},\n_v]=e^{\rho\phi}[\p_{\b{v}}+\rho\phi_{\b{v}},\p_v- k\phi_v]=-\rho-k,	
\end{align}
and
\begin{align}\label{1.3}
\begin{split}
	[\n_{l\b{v}}, e^{\rho\phi}\n_v]&=[\n_{l\b{v}}, e^{\rho\phi}]\n_v+e^{\rho\phi}[\n_{l\b{v}},\n_v]\\
	&=e^{\rho\phi}\rho\phi_{\b{v}}\n_v+(-l\rho-k)\\
	&=e^{\rho\phi}\n_v \rho\phi_{\b{v}}+(-(l+1)\rho-k).
\end{split}	
\end{align}
Combining  (\ref{notation1}) and (\ref{1.3}) we find
\begin{align}\label{0.2}
\begin{split}
\n_{\b{v}}e^{\rho\phi}\n_{v}\n_{\b{v}}&=(e^{\rho\phi}\n_{v}\n_{\b{v}}+e^{\rho\phi}\n_v \rho\phi_{\b{v}}+(-2\rho-k))\n_{\b{v}}\\
&=e^{\rho\phi}\n_{v}\n_{2\b{v}}\n_{\b{v}}+(-2\rho-k)\n_{\b{v}}.
\end{split}
\end{align}
By  induction we get 
\begin{align}\label{0.3}
\n_{l\b{v}}\cdots\n_{\b{v}}(e^{\rho\phi}\n_{v})\n_{\b{v}}=(e^{\rho\phi}\n_{v})\n_{(l+1)\b{v}}\cdots\n_{\b{v}}+\left(-k-\frac{l+3}{2}\rho\right)l\n_{l\b{v}}\cdots\n_{\b{v}}.
\end{align}
for any $l\ge 1$. In fact for $l=1$ this is exactly (\ref{0.2}). So we
assume that (\ref{0.3}) holds for $1, \cdots,  l-1$. 
By (\ref{1.3})
\begin{align*}
	&\n_{l\b{v}}\cdots\n_{\b{v}}(e^{\rho\phi}\n_{v})\n_{\b{v}}=\n_{l\b{v}}\left((e^{\rho\phi}\n_{v})\n_{l\b{v}}\cdots\n_{\b{v}}+(-k-\frac{l+2}{2}\rho)(l-1)\n_{(l-1)\b{v}}\cdots\n_{\b{v}}\right)\\
	&=(e^{\rho\phi}\n_{v})\n_{(l+1)\b{v}}\cdots\n_{\b{v}}+\left((-k-(l+1)\rho)+(-k-\frac{l+2}{2}\rho)(l-1)\right)\n_{l\b{v}}\cdots\n_{\b{v}}\\
	&=(e^{\rho\phi}\n_{v})\n_{(l+1)\b{v}}\cdots\n_{\b{v}}+(-k-\frac{l+3}{2}\rho)l\n_{l\b{v}}\cdots\n_{\b{v}},
\end{align*}
 completing the proof of (\ref{0.3}). 

For later convenience we set
\begin{align}\label{set1}
k=\frac{m-\rho}{2},\quad A_n=\left(-k-\frac{n+3}{2}\rho\right)n=\left(-\frac{m}{2}-\frac{n+2}{2}\rho\right)n
\end{align}
and 
\begin{align}\label{set2}
\square_{n}:=\underbrace{(e^{\rho\phi}\n_{v})\cdots(e^{\rho\phi}\n_{v})}_n\n_{n\b{v}}\n_{(n-1)\b{v}}\cdots\n_{\b{v}}.	
\end{align}
The formula (\ref{0.3}) can now be written as
\begin{align}\label{2.1}
\square_{n}\square_{1}=\square_{n+1}+A_n\square_{n},\quad 	\square_{1}=(e^{\rho\phi}\n_{v})\n_{\b{v}}.
\end{align}

\begin{lemma}\label{lemma3}
For any $N\in \mb{N}_+$,  the operator $\square_{N+1}$ is self-adjoint and 
\begin{align}
(\square_{N+1})^2=\sum_{n=N+1}^{2N+2}b_n \square_{n}	
\end{align}
	for some constants $b_n$ with $|b_n|=O(k^{N+1})$ as $k\to\infty$.
\end{lemma}
\begin{proof}
	Using (\ref{2.1}) we have
the following factorization
	\begin{align}\label{1.4}
		\square_{N+1}=(\square_{1}-A_N)\square_{N}=\prod_{n=0}^N (\square_{1}-A_n).
	\end{align}
The identity (\ref{0.1}) implies that
\begin{align}
\square_{1}=k+\rho-\n'^*\n'
\end{align}
is self-adjoint. Combining this with (\ref{1.4}) 
we see  that $\square_{N+1}$ is also self-adjoint.

We compute 
	$(\square_{N+1})^2$ using (\ref{1.4}) and (\ref{2.1}),
\begin{align}\label{1.5}
\begin{split}
	(\square_{N+1})^2
&=\prod_{n=0}^N (\square_{1}-A_n)\square_{N+1}\\
	&=\prod_{n=0}^{N-1}(\square_{1}-A_n)\left((\square_{1}-A_{N+1})+(A_{N+1}-A_N)\right)\square_{N+1}\\
	&=\prod_{n=0}^{N-1}(\square_{1}-A_n)\square_{N+2}+(A_{N+1}-A_N)\prod_{n=0}^{N-1}(\square_{1}-A_n)\square_{N+1}\\
	&=\prod_{n=0}^{N-1}(\square_{1}-A_n)
        \square_{N+2}+O(k)\prod_{n=0}^{N-1}(\square_{1}-A_n)\square_{N+1},
	\end{split}
\end{align}
where the last equality follows from (\ref{set1}), $|A_n|=O(k)$.
In other words
$\prod_{n=0}^N (\square_{1}-A_n)\square_{N+1}$ is a linear combination
of of $\square_{N+2}$ and $\square_{N+1}$ with the coefficients
being of the form $\prod_{n=0}^{N-1}(\square_{1}-A_n)$.
Repeating the same procedure we can reduce
the coefficients to be constants, namely there exist constants $b_n$ with $|b_n|\in O(k^{N+1}), N+1\leq n\leq 2N+2$ such that
\begin{align*}
	(\square_{N+1})^2=\sum_{n=N+1}^{2N+2}b_n \square_{n}.	
\end{align*}

\end{proof}

\begin{lemma}\label{2.2}
 For any $N\in \mb{N}_+$ and $m>(-\rho)N$
the following identity 
holds as an operator
 on the space $A^{0,1}(\mc{X}_y, L^k)$, 
\begin{align}
(k+\Delta')^{-1}(1-a_N\square_{N+1})=\sum_{n=0}^{N}a_n \square_{n},
\end{align}
	where $\square_{0}:=\text{Id}$ and
	\begin{align}\label{defnan}
	a_n=\frac{2^{n+1}}{(n+2)!(m+\rho n)\cdots(m+\rho)m}.	
	\end{align}
\end{lemma}
\begin{proof}
By (\ref{0.1}) and (\ref{2.1}) we have
\begin{align}
k-\Delta'+\rho=k-\n'^*\n'+\rho=e^{\rho\phi}\n_{v}\n_{\b{v}}=\square_{1}.
\end{align}
Recalling the notation in  (\ref{set1}) this becomes
\begin{align}\label{2.18}
k+\Delta'=2k+\rho-\square_{1}=m-\square_{1}.
\end{align}
A direct computation using (\ref{2.1}) gives
\begin{align*}
&\quad (\sum_{n=0}^{N}a_n\square_{n})(k+\Delta')=\sum_{n=0}^{N}a_n\square_{n}(m-\square_{1})\\
&=\sum_{n=0}^{N} ma_n\square_{n}-\sum_{n=0}^{N}a_n(\square_{n+1}+A_n\square_{n})\\
&=ma_0+\sum_{n=1}^{N} \left((m-A_n)a_n-a_{n-1}\right)\square_{n}-a_N\square_{N+1}\\
&=1-a_N\square_{N+1},	
\end{align*}
where the last equality holds since $a_0=\frac{1}{m}$ and 
\begin{align*}
	(m-A_n)a_n-a_{n-1}=\frac{1}{2}(n+2)(m+\rho n)a_n-a_{n-1}=0.
\end{align*}
\end{proof}
Denote 
\begin{align}\label{2.3}
\|\o{\n}^n\mu\|^2:=\langle\n_{n\b{v}}\n_{(n-1)\b{v}}\cdots \n_{\b{v}}\mu,e^{n\rho\phi}\n_{n\b{v}}\n_{(n-1)\b{v}}\cdots \n_{\b{v}}\mu\rangle	
\end{align}
for $n\geq 1$, and $\|\o{\n}^0\mu\|^2:=\|\mu\|^2$. Then
\begin{lemma}\label{2.5} For $n\geq 0$ and $\mu=-\p_{\b{v}}(\phi_{\alpha\b{v}}(\phi_{v\b{v}})^{-1})\zeta^{\alpha} d\b{v}\otimes \frac{\p}{\p v}$ (see also (\ref{definition of mu})) we have the following identity
	\begin{align}
		\|\o{\n}^n\mu\|^2=(-\rho)^n\frac{n!(n+3)!}{3\cdot 2^{n+1}}\|\mu\|^2.
	\end{align}
\end{lemma}
\begin{proof} We observe first that the adjoint 
$(e^{\rho \phi} \nabla_{\bar v})^* =-
e^{\rho \phi}( \partial_v -\rho \phi_v )
$, and  with some abuse of notation we introduce temperarily
$D_{-v}=
 \partial_v -\rho \phi_v $. Thus
	\begin{align*}
	\begin{split}
		\|\o{\n}^n\mu\|^2
		&=\langle\n_{n\b{v}}\n_{(n-1)\b{v}}\cdots \n_{\b{v}}\mu,e^{\rho\phi}\n_{\b{v}}e^{(n-1)\rho\phi}\n_{(n-1)\b{v}}\cdots \n_{\b{v}}\mu\rangle\\
		&=-\langle e^{\rho\phi} D_{-v}\n_{n\b{v}}\n_{(n-1)\b{v}}\cdots \n_{\b{v}}\mu,e^{(n-1)\rho\phi}\n_{(n-1)\b{v}}\cdots \n_{\b{v}}\mu\rangle.
		\end{split}
	\end{align*}
We  use Lemma \ref{lemma1} repeatedly
and find
	\begin{align}
	\begin{split}
\label{2.4}
		\|\o{\n}^n\mu\|^2
		&=(-\rho)(n+1)\langle\n_{(n-1)\b{v}}\cdots \n_{\b{v}}\mu,e^{(n-1)\rho\phi}\n_{(n-1)\b{v}}\cdots \n_{\b{v}}\mu\rangle\\
		&\quad -\langle e^{\rho\phi}\n_{n\b{v}}D_{-v}\n_{(n-1)\b{v}}\cdots \n_{\b{v}}\mu,e^{(n-1)\rho\phi}\n_{(n-1)\b{v}}\cdots \n_{\b{v}}\mu\rangle\\
		&=(-\rho)(n+1+n+\cdots+2)\|\o{\n}^{n-1}\mu\|^2\\
		&\quad-\langle e^{\rho\phi}\n_{n\b{v}}\n_{(n-1)\b{v}}\cdots
 \n_{\b{v}}D_{-v}\mu,e^{(n-1)\rho\phi}\n_{(n-1)\b{v}}\cdots \n_{\b{v}}\mu\rangle\\
		&=(-\rho)\frac{n(n+3)}{2}\|\o{\n}^{n-1}\mu\|^2,
	\end{split}
	\end{align}
where  the last equality holds by a direct checking and using (\ref{1.1}), namely
\begin{align}\label{2.12}
D_{-v}\mu=(\p_v-\rho\phi_v)\left(-\p_{\b{v}}(\phi_{\alpha\b{v}}(\phi_{v\b{v}})^{-1})\zeta^{\alpha}\right)=0.
\end{align}
Hence
\begin{align}
\|\o{\n}^n\mu\|^2=(-\rho)^n\frac{n!(n+3)!}{3\cdot 2^{n+1}}\|\mu\|^2.	
\end{align}
\end{proof}
We shall also
need the following elementary  identity
involving
the Gamma function $\Gamma(x)$.
\begin{lemma}\label{lemma2}
  Let  $0<a<\frac{1}{N+1}$. The following
  summation formula holds
\begin{align*}
\begin{split}
	&\quad\sum_{n=0}^N (-1)^nn!(n+3)\Gamma(\frac{1}{a}-n)
	=\frac{\Gamma(\frac{1}{a})(5a+3)}{(2a+1)(a+1)}\\
	&+\frac{(-1)^{N+1}(N+1)!\Gamma(\frac{1}{a}-(N+1))((N+6)a+(N+4))(a(N+1)-1)}{(2a+1)(a+1)}.
	\end{split}
\end{align*}	
\end{lemma}
\begin{proof}
  This is simply a consequence of the following
  identity
	\begin{align*}
	\begin{split}
	&\quad(2a+1)(a+1)(n+3)\\
	&=-((n+5)a+(n+3))(an-1)-((n+6)a+(n+4))(a(n+1)-1)\frac{n+1}{\frac{1}{a}-(n+1)}.	
	\end{split}
	\end{align*}
Indeed
\begin{align*}
\begin{split}
	&\quad (2a+1)(a+1)\sum_{n=0}^N (-1)^nn!(n+3)\Gamma(\frac{1}{a}-n)\\
	&=\sum_{n=0}^N -(-1)^nn!\Gamma(\frac{1}{a}-n)((n+5)a+(n+3))(an-1)\\
	&\quad+\sum_{n=0}^N (-1)^{n+1}n!\Gamma(\frac{1}{a}-n)((n+6)a+(n+4))(a(n+1)-1)\frac{n+1}{\frac{1}{a}-(n+1)}\\
	&=-\sum_{n=0}^N (-1)^nn!\Gamma(\frac{1}{a}-n)((n+5)a+(n+3))(an-1)\\
	&\quad+\sum_{n=0}^N (-1)^{n+1}(n+1)!\Gamma(\frac{1}{a}-(n+1))((n+6)a+(n+4))(a(n+1)-1)\\
	&=\Gamma(\frac{1}{a})(5a+3)\\
	&\quad+(-1)^{N+1}(N+1)!\Gamma(\frac{1}{a}-(N+1))((N+6)a+(N+4))(a(N+1)-1),\\
\end{split}	
\end{align*}
completing  the proof.
\end{proof}

We return now to (\ref{1.2}).
From (\ref{rho}) we have
$(-\rho)c_1(L|_{\mc{X}_y})=c_1(K_{\mc{X}_y}),
$
and then $\rho<0$ since both $L|_{\mc{X}_y}$ and $K_{\mc{X}_y}$ are ample. Let $h$ be a smooth metric on $K_{\mc{X}_y}$ such that its curvature $R(h)=\sqrt{-1}(-\rho)\phi_{v\b{v}}dv\wedge d\b{v}$, which gives a K\"ahler metric on $\mc{X}_y$. Moreover, by a direct calculation, its scalar curvature is $-1$.
Using (\ref{metric consist}) and
 Theorem \ref{TYZ} we get the following TYZ expansion associated with the line bundle $(L^k+K_{\mc{X}/\mc{T}})|_{\mc{X}_y}$,
\begin{align}\label{2.6}
\begin{split}
\sum_{j=1}^{d_k}|u_j|^2e^{-k\phi}&=(-\rho)\left((-\frac{k}{\rho}+1)+\frac{-1}{2}+O\left(e^{-\frac{(\log (-\frac{k}{\rho}+1))^2}{8}}\right)\right)\frac{\omega|_{\mc{X}_y}}{2\pi}\\
&=\left(k+\frac{-\rho}{2}+O\left(e^{-\frac{(\log (-\frac{k}{\rho}+1))^2}{8}}\right)\right)\frac{\omega|_{\mc{X}_y}}{2\pi}.
\end{split}
\end{align}

Now we prove Theorem  \ref{main theorem 1}.
\begin{proof}
 For any $l\geq 0$, we fix an integer $N\in \mb{N}_+$ such that
$
\frac{N-1}{2}>l,	
$
and  take $m$ large such that $m>(-\rho)(N+1)$.

The second term in the RHS of (\ref{1.2})
is by Lemma \ref{2.2} 
\begin{align}\label{2.19}
\begin{split}
\sum_{j=1}^{d_k}k\langle (k+\Delta')^{-1}i_{\mu}u_j,i_{\mu}u_j\rangle
=&\sum_{j=1}^{d_k}k\langle \sum_{n=0}^{N}a_n\square_{n} i_{\mu}u_j,i_{\mu}u_j\rangle\\
&+\sum_{j=1}^{d_k}k\langle (k+\Delta')^{-1}a_N\square_{N+1} i_{\mu}u_j,i_{\mu}u_j\rangle.
\end{split}
\end{align}

We compute the first term in the RHS of the above formula as
\begin{align}\label{3.0-}
	\begin{split}
		&\quad\sum_{j=1}^{d_k}k\langle \sum_{n=0}^{N}a_n\square_{n} i_{\mu}u_j,i_{\mu}u_j\rangle\\
		&=\sum_{n=0}^{N}\sum_{j=1}^{d_k}k a_n(-1)^n\langle i_{\n_{n\b{v}}\cdots\n_{\b{v}}\mu}u_j, (e^{\rho\phi}\n_{\b{v}})^n i_{\mu}u_j\rangle\\
		&=\sum_{n=0}^{N}\sum_{j=1}^{d_k}k a_n(-1)^n\langle i_{\n_{n\b{v}}\cdots\n_{\b{v}}\mu}u_j, e^{n\rho\phi}i_{\n_{n\b{v}}\cdots\n_{\b{v}}\mu}u_j\rangle\\
		&=\sum_{n=0}^{N}\sum_{j=1}^{d_k}k a_n(-1)^n\int_{\mc{X}/\mc{T}}|\o{\n}^n\mu|^2|u_j|^2e^{-k\phi},
	\end{split}
\end{align}
where we use (\ref{set2}) and Stoke's theorem
in the first equality,
 a direct computation and (\ref{notation1})
in the second equality.
This combined with (\ref{2.6}) gives
\begin{align}\label{3.0}
	\begin{split}
		&\quad\sum_{j=1}^{d_k}k\langle \sum_{n=0}^{N}a_n\square_{n} i_{\mu}u_j,i_{\mu}u_j\rangle\\
				&=\sum_{n=0}^{N}k a_n(-1)^n\|\o{\n}^n\mu\|^2\left(k+\frac{-\rho}{2}+O\left(e^{-\frac{(\log (-\frac{k}{\rho}+1))^2}{8}}\right)\right)\frac{1}{2\pi}.
	\end{split}
\end{align}
Disregarding the remainder terms and using Lemma \ref{2.5} and
(\ref{defnan})
we get
\begin{align}\label{2.7}
\begin{split}
	&\quad\sum_{n=0}^{N}k a_n(-1)^n\frac{\|\o{\n}^n\mu\|^2}{\|\mu\|^2}\frac{1}{2\pi}(k+\frac{-\rho}{2})\\
	&=\sum_{n=0}^{N}k\frac{2^{n+1}(-1)^n}{(n+2)!(m+\rho n)\cdots(m+\rho)m}\cdot (-\rho)^n\frac{n!(n+3)!}{3\cdot 2^{n+1}}\frac{1}{2\pi}(k+\frac{-\rho}{2})\\
	&=\frac{m}{24\pi}\sum_{n=0}^N (-1)^n(\frac{-\rho}{m})^n\frac{n!(n+3)}{(1+\frac{\rho}{m}n)\cdots(1+\frac{\rho}{m})}(1+\frac{-2\rho}{m})(1+\frac{-\rho}{m})\\
	&=\frac{m}{24\pi}(1+\frac{-\rho}{m})(1+\frac{-2\rho}{m})\frac{1}{\Gamma(\frac{m}{-\rho})}\sum_{n=0}^N (-1)^nn!(n+3)\Gamma(\frac{m}{-\rho}-n).
\end{split}
\end{align}
From Lemma \ref{lemma2} and noting $\frac{m}{-\rho}>N+1$, the above
 can be  written as
\begin{align}\label{2.22}
\begin{split}
	&\quad \sum_{n=0}^{N}k a_n(-1)^n\frac{\|\o{\n}^n\mu\|^2}{\|\mu\|^2}\frac{1}{2\pi}(k+\frac{-\rho}{2})=\frac{m}{24\pi}(3+5\frac{-\rho}{m})\\
	&+\frac{m}{24\pi}(-1)^{N+1}(N+1)!\frac{((N+4)+(N+6)(\frac{-\rho}{m}))(\frac{-\rho}{m}(N+1)-1)}{(1+\frac{\rho}{m})\cdots(1+(N+1)\frac{\rho}{m})}(\frac{-\rho}{m})^{N+1}\\
	&=\frac{m}{24\pi}(3+5\frac{-\rho}{m})+O(\frac{1}{m^{N}})\\
 &=\frac{3k-\rho}{12\pi}+O(\frac{1}{k^{N}}).
		\end{split}
\end{align}
On the other hand, it is easy to see that
\begin{align}\label{2.21}
	\sum_{n=0}^{N}k a_n(-1)^n\|\o{\n}^n\mu\|^2O\left(e^{-\frac{(\log (-\frac{k}{\rho}+1))^2}{8}}\right)\frac{1}{2\pi}=o(\frac{1}{m^N})=o(\frac{1}{k^N}).
\end{align}
Substituting (\ref{2.22}) and (\ref{2.21}) into (\ref{3.0}), we get
\begin{align}\label{2.23}
	\sum_{j=1}^{d_k}k\langle \sum_{n=0}^{N}a_n\square_{n} i_{\mu}u_j,i_{\mu}u_j\rangle=\frac{3k-\rho}{12\pi}\|\mu\|^2+O(\frac{1}{k^{N}}).
\end{align}

For the second term in the RHS of (\ref{2.19}) we
use Cauchy-Schwarz inequality,
\begin{align}\label{2.24}
\begin{split}
	&\quad\sum_{j=1}^{d_k}k\langle (k+\Delta')^{-1}a_N\square_{N+1} i_{\mu}u_j,i_{\mu}u_j\rangle\\
	&\leq \sum_{j=1}^{d_k}k|a_N|\|(k+\Delta')^{-1}\square_{N+1} i_{\mu}u_j\|\|i_{\mu}u_j\|\\
	&\leq \sum_{j=1}^{d_k}|a_N|\|\square_{N+1} i_{\mu}u_j\|\|i_{\mu}u_j\|,
	\end{split}	
\end{align}
where the second inequality holds because the eigenvalues of $k+\Delta'$ are greater
than or equal to $k$. By Lemma \ref{lemma3}, $\square_{n+1}$ is
self-adjoint, so that
\begin{align}\label{1.6}
\begin{split}
	\|\square_{N+1} i_{\mu}u_j\|&=|\langle(\square_{N+1})^2i_{\mu}u_j,i_{\mu}u_j\rangle|^{1/2}\\
	&=\left|\sum_{n=N+1}^{2N+2}b_n\langle \square_{n}i_{\mu}u_j,i_{\mu}u_j\rangle\right|^{1/2}\\
	&=\left|\sum_{n=N+1}^{2N+2}b_n(-1)^n\int_{\mc{X}/\mc{T}}|\o{\n}^n\mu|^2|u_j|^2 e^{-k\phi}\right|^{1/2}\\
	&\leq \left(\sum_{n=N+1}^{2N+2}|b_n|\max_{\mc{X}_y}|\o{\n}^n\mu|^2\|u_j\|^2\right)^{1/2}=O\left(k^{\frac{N+1}{2}}\right),
\end{split}	
\end{align}
where the second equality follows from Lemma \ref{lemma3}, the third
equality holds by the same proof as (\ref{3.0}), the last equality
holds since $|b_n|=O(k^{N+1})$, $\|u_j\|=1$ and
$\max_{\mc{X}_y}|\o{\n}^n\mu|^2$ is independent of $k$. Substituting
(\ref{1.6}) into (\ref{2.24})
gives 
\begin{align}
\begin{split}
	&\quad\sum_{j=1}^{d_k}k\langle (k+\Delta')^{-1}a_N\square_{N+1} i_{\mu}u_j,i_{\mu}u_j\rangle\\
		&\leq \sum_{j=1}^{d_k}|a_N|\|i_{\mu}u_j\|O\left(k^{\frac{N+1}{2}}\right)\leq |a_N|d_k\max_{\mc{X}_y}|\mu|^2O\left(k^{\frac{N+1}{2}}\right)\\
		&=O\left(\frac{1}{k^{\frac{N-1}{2}}}\right)
	\end{split}	
\end{align}
where the last equality follows from
$|a_N|=O\left(\frac{1}{k^{N+1}}\right)$, $d_k=O(k)$, and
$\max_{\mc{X}_y}|\mu|^2$ is independent of $k$.
This settles the second term in the RHS of (\ref{2.19}).

Combining the above estimate with   (\ref{2.23})
we obtain
the estimates for  (\ref{2.19}),
\begin{align}\label{2.10}
	\sum_{j=1}^{d_k}k\langle (k+\Delta')^{-1}i_{\mu}u_j,i_{\mu}u_j\rangle=\frac{3k-\rho}{12\pi}\|\mu\|^2+O\left(\frac{1}{k^{\frac{N-1}{2}}}\right).
\end{align}

On the other hand by \cite[(3.46)]{Wan} and (\ref{dec}) we have
\begin{align}
\begin{split}
\Delta(c(\phi)(\zeta,\b{\zeta}))&=\phi_{v\b{v}}^{-1}\frac{\p^2}{\p v\p\b{v}}(c(\phi)_{\alpha\b{\beta}})	\zeta^{\alpha}\b{\zeta}^{\beta}\\
&=(\p\b{\p}\log \phi_{v\b{v}})(\zeta^{\alpha}\frac{\delta}{\delta z^{\alpha}},\o{\zeta^{\beta}\frac{\delta}{\delta z^{\beta}}})-|\mu|^2\\
&=(-\rho)\p\b{\p}\phi(\zeta^{\alpha}\frac{\delta}{\delta z^{\alpha}},\o{\zeta^{\beta}\frac{\delta}{\delta z^{\beta}}})-|\mu|^2\\
&=(-\rho)(-\sqrt{-1})c(\phi)(\zeta,\b{\zeta})-|\mu|^2.
\end{split}
\end{align}
By integrating along fibers and using Stoke's theorem, we get
\begin{align}\label{2.17}
(-\sqrt{-1})\left(\int_{\mc{X}_y}c(\phi)\omega\right)(\zeta,\b{\zeta})=\frac{1}{-\rho}\|\mu\|^2.	
\end{align}

From above equality, the first term in the RHS of (\ref{1.2}) is 
\begin{align}\label{2.11}
\begin{split}
	&\quad -\frac{\sqrt{-1}}{2\pi}\left(\int_{\mc{X}_y}kc(\phi)\sum_{j=1}^{d_k}|u_j|^2e^{-k\phi}\right)(\zeta,\b{\zeta})\\
	&=-\frac{\sqrt{-1}}{2\pi}\left(\int_{\mc{X}_y}c(\phi)\omega\right)(\zeta,\b{\zeta})k \frac{1}{2\pi}\left(\frac{k}{\rho^2}+\frac{1}{2(-\rho)}+O\left(e^{-\frac{(\log (-\frac{k}{\rho}+1))^2}{8}}\right)\right)\\
	&=\frac{1}{2\pi}\frac{\|\mu\|^2}{-\rho}k \frac{1}{2\pi}\left(k+\frac{-\rho}{2}+O\left(e^{-\frac{(\log (-\frac{k}{\rho}+1))^2}{8}}\right)\right)\\
	&=\frac{k\rho-2k^2}{8\pi^2\rho}\|\mu\|^2+O\left(ke^{-\frac{(\log (-\frac{k}{\rho}+1))^2}{8}}\right).
	\end{split}
\end{align}
Finally using  (\ref{2.10}) and (\ref{2.11}) we get the an expansion
for (\ref{1.2}) as
\begin{align}\label{asymptotic 1}
\begin{split}
	&\quad(-\sqrt{-1})c_1(E^k, \|\bullet\|_k)(\zeta,\b{\zeta})\\
	&=\frac{3k-\rho}{24\pi^2}\|\mu\|^2+O(\frac{1}{k^{\frac{N-1}{2}}})+\frac{k\rho-2k^2}{8\pi^2\rho}\|\mu\|^2+O\left(ke^{-\frac{(\log (-\frac{k}{\rho}+1))^2}{8}}\right)\\
	&=\frac{6k^2-6k\rho+\rho^2}{24\pi^2(-\rho)}\|\mu\|^2+O\left(\frac{1}{k^{\frac{N-1}{2}}}\right).
	\end{split}
\end{align}
This
completes the proof of
Theorem 
\ref{main theorem 1}
since $\frac{N-1}{2}>l$.
\end{proof}

\subsection{The curvature of Quillen metric} We prove now Theorem \ref{main theorem 2}.
\begin{proof}
From \cite[Proposition 3.9]{Wan}, the curvature of Quillen metric is 
\begin{align}\label{asymptotic 2}
\begin{split}
&\quad (\sqrt{-1}c_1(\lambda,\|\bullet\|_Q))(\zeta,\b{\zeta})\\
&=\frac{k^{2}}{(2\pi)^2}\int_{\mc{X}_y}(-\sqrt{-1})c(\phi)(\zeta,\b{\zeta})\omega+\frac{k}{(2\pi)^2}\int_{\mc{X}_y}\left(\frac{1}{2}|\mu|^2-\frac{\rho}{2}(-\sqrt{-1})c(\phi)(\zeta,\b{\zeta})\right)\omega\\
	&\quad+\frac{1}{(2\pi)^{2}}\int_{\mc{X}_y}\left((-\sqrt{-1})c(\phi)(\zeta,\b{\zeta})\left(-\frac{1}{6}\Delta\rho+\frac{1}{24}(|R|^2-4|Ric|^2+3\rho^2\right)-\frac{\rho}{4}|\mu|^2\right)\omega\\
	&\quad +\frac{1}{(2\pi)^{2}}\left(\frac{1}{12}\|\mu\|^2_{Ric}+\frac{1}{12}\|\n'\mu\|^2-\frac{1}{4}\|\b{\p}^*\mu\|^2\right),
\end{split}
\end{align}
where $\|\mu\|^2_{Ric}:=\int_{\mc{X}_y}(\mu^i_{\b{j}}\o{\mu^s_{\b{t}}}R_{i\b{l}}\phi^{\b{l}k}\phi^{\b{j}t}\phi_{k\b{s}})\frac{\omega^n}{n!}$.

The Chern curvature tensor $R$ is given by  
\begin{align}
R_{v\b{v}v\b{v}}=-\p_v\p_{\b{v}}\phi_{v\b{v}}+\phi_{v\b{v}}^{-1}\p_v\phi_{v\b{v}}\p_{\b{v}}\phi_{v\b{v}}=\rho(\phi_{v\b{v}})^2,
\end{align}
and then
\begin{align}\label{2.13}
|R|^2=|R_{v\b{v}v\b{v}}|^2(\phi_{v\b{v}})^{-4}=\rho^2,\quad Ric_{v\b{v}}=(\phi_{v\b{v}})^{-1}R_{v\b{v}v\b{v}}=\rho\phi_{v\b{v}}.
\end{align}
Hence
\begin{align}\label{2.14}
\|\mu\|^2_{Ric}=\rho\|\mu\|^2, \quad |Ric|^2=\rho^2.	
\end{align}
Using  (\ref{2.12}) we have
\begin{align}\label{2.15}
\n'\mu=(\p_v-\rho\phi_v)\mu dv=D_{-v}\mu dv=0,
\end{align}
consequently
\begin{align}\label{2.16}
\b{\p}^*\mu=-\sqrt{-1}\Lambda\n'\mu=0.	
\end{align}

Substituting (\ref{2.13}), (\ref{2.14}), (\ref{2.15}) and (\ref{2.16})
into (\ref{asymptotic 2}),
and using (\ref{2.17}), we obtain
\begin{align}
(\sqrt{-1}c_1(\lambda,\|\bullet\|_Q))(\zeta,\b{\zeta})=\frac{6k^2-6k\rho+\rho^2}{24\pi^2(-\rho)}\|\mu\|^2,
\end{align}
completing the proof of
Theorem \ref{main theorem 2}.
\end{proof}

\subsection{The proof of Theorem \ref{main theorem}}

\begin{proof}
  We estimate (\ref{Torsion})
using (\ref{asymptotic 1}) and (\ref{asymptotic 2}), 
\begin{align*}
\begin{split}
\frac{1}{2\pi}\p\b{\p}\log(\tau_k(\b{\p}))^2(\zeta,\b{\zeta})
&=(-\sqrt{-1})(-c_1(\lambda,\|\bullet\|_Q)-c_1(E^k, \|\bullet\|_k))(\zeta,\b{\zeta})\\
&=O\left(\frac{1}{k^{\frac{N-1}{2}}}\right)=o(k^{-l}).\\
\end{split}
\end{align*}
\end{proof}


\begin{thebibliography}{99}



\bibitem{Ahlfors}
L. Ahlfors,  {\it
  Some Remarks on Teichmuller's Space of Riemann Surfaces
  },
{Ann. Math.},
  {\bf 74} (1961),
{171-191}.


\bibitem{Bern2} B. Berndtsson, {\it Curvature of vector bundles associated to holomorphic fibrations}, Ann. Math. {\bf 169} (2009), 531-560.

\bibitem{Bern3} B. Berndtsson, {\it Positivity of direct image bundles and convexity on the space of K\"{a}hler metrics}, J. Differ. Geom. {\bf 81} (2009), no.3, 457-482.

\bibitem{Bern4} B. Berndtsson, {\it Strict and non strict positivity of direct image bundles}, Math. Z. {\bf 269} (2011), 1201-1218.

\bibitem{BGV} N. Berline, E. Getzler, M. Vergne, Heat Kernels and Dirac Operators, Springer-Verlag Berline Heidelberg, 1996.


\bibitem{Bismut1} J.-M. Bismut, H. Gillet, C. Soul\'{e}, {\it Analytic torsion and holomorphic determinant bundles I. Bott-Chern forms and analytic torison}, Commun. Math. Phy. {\bf 115} (1988), 49-78.
\bibitem{Bismut2} J.-M. Bismut, H. Gillet, C. Soul\'{e}, {\it Analytic torsion and holomorphic determinant bundles II. Direct images and Bott-Chern forms}, Commun. Math. Phy. {\bf 115} (1988), 79-126.
\bibitem{Bismut3} J.-M. Bismut, H. Gillet, C. Soul\'{e}, {\it Analytic torsion and holomorphic determinant bundles III. Quillen metrics on holomorphic determinants}, Commun. Math. Phy. {\bf 115} (1988), 301-351.





\bibitem{Englis} M. Englis,
{\it Asymptotics of reproducing kernels on a plane domain},
Proc. Amer. Math. Soc. {\bf 123} (1995), no. 10, 3157-3160. 


\bibitem{Fay} J. Fay, Kernel functions, analytic torsion, and moduli spaces, Mem. Amer. Math. Soc.  {\bf 96} (1992), no. 464, vi+123.

\bibitem{Feng} H. Feng, K. Liu, X. Wan, {\it Geodesic-Einstein metrics and nonlinear stabilities}, arXiv: 1710.10243v1, (2017).

\bibitem{Zhang} K. Fedosova, J. Rowlett, G. Zhang, {\it Second variation of Selberg zeta functions and curvature asymptotics}, arXiv:1709.03841, (2017).

\bibitem{Finski} S. Finski, {\it On the full asymptotic of analytic torsion}, arXiv: 1705. 02779v1, (2017).

\bibitem{Liu} C.-J. Liu, {\it The asymptotic Tian-Yau-Zelditch expansion on Riemann surfaces with constant curvature}, Taiwanese Journal of Mathematics, Vol. {\bf 14} (2010), No. 4, 1665-1675.


\bibitem{Ma1} X. Ma, G. Marinescu, Holomorphic Morse Inequalities and Bergman Kernels, Birkh\"auser, Basel$\cdot$ Boston$\cdot$ Berlin, 2006.


\bibitem{Ray} D. B. Ray, I. M. Singer, {\it Analytic torsion for complex manifolds}, Annals of Mathematics {\bf 98} (1973), no. 1, 154-177.

\bibitem{Sarnak} P. Sarnak, {\it Determinants of Laplacians}, Communications in Mathematical Physics {\bf 110} (1987), no. 1, 113-120.








 \bibitem{Wan} X. Wan, G. Zhang, {\it The asymptotic of curvature of direct image bundle associated with higher powers of a relative ample line bundles}, arXiv: 1712.05922v1, (2017). 


\end{thebibliography}
\end{document}